\documentclass[12pt, leqno]{amsart}

\usepackage{amsmath,amssymb,amsthm}
\usepackage{euscript, enumitem}

\usepackage{url}

\usepackage{fouridx}%

\newtheorem{Thm}{Theorem}[section]

\newtheorem{Cor}{Corollary}[section]
\newtheorem{Prop}{Proposition}[section]
\newtheorem{Lem}{Lemma}[section]

\newtheorem{Exa}{Example}[section]

\newcommand{\Top}{\ensuremath{\mathfrak{T}}}

\newcommand{\SU}{\ensuremath{\EuScript U}}
\newcommand{\SUa}[1]{\ensuremath{\EuScript U_{#1}\ }}
\newcommand{\PS}{\ensuremath{\EuScript P}}

\renewcommand{\star}[1]{\ensuremath{{\fourIdx{*}{}{}{}{#1}}}}

\newcommand{\st}{\ensuremath{\mathfrak{st}}}

\newcommand{\R}{\ensuremath{\mathbb R}}

\newcommand{\N}{\ensuremath{\mathbb N}}
\newcommand{\Z}{\ensuremath{\mathbb Z}}

\newcommand{\compl}[1]{\ensuremath{#1^\complement}}

\title{Egyptian fractions on groups}

\author{David A. Ross}
\address{Department of Mathematics, University of Hawai'i, Honolulu, Hawai'i 96822, USA}
\email{ross@math.hawaii.edu}

\date{}

\begin{document}

\maketitle


\renewcommand{\thefootnote}{}



\renewcommand{\thefootnote}{\arabic{footnote}}
\setcounter{footnote}{0}

\begin{abstract}
Results about the structure of the set of Egyptian fractions on the line are extended to subsets of topological groups.
\end{abstract}

\section{Introduction}

An \emph{Egyptian fraction} is a sum of unitary fractions
\[\frac{1}{m_1}+\cdots+\frac{1}{m_n}\](where $n, m_i\in\N$).
These were introduced into western mathematics by Fibonacci in 1202 \cite{pisa}, and have been a rich source of mathematical problems ever since.  (For a survey, see David Eppstein's excellent website \cite{Eppstein}.)

Inspired by results of Kellogg~\cite{Kellogg} and Curtiss~\cite{Curtiss}, in 1956 Sierpinski~\cite{Sierpinski} published several results about the structure of the set of Egyptian fractions.  Recently Nathanson~\cite{Nath} has extended these results to more general set of real numbers, and the author~\cite{Ross} showed that working with field extensions of $\R$ makes it possible to simplify and extend Sierpinski's results.  A similar approach was adopted by G\"{o}ral and Sertba\c{s}~\cite{Goral}.

In this note we use the methods of \cite{Ross} to extend these results to topological groups.  The extension is interesting not only because of the increased abstraction possible, but also because the proofs from the earlier papers, which rely heavily on the linear order and metric on $\R$, do not apply in this setting.

\subsection{Note on the proofs}

The proofs in this paper were discovered using nonstandard analysis \cite{Davis, Edinburgh}, and we retain the notation of that technology.  However, the amount of nonstandard analysis used is very slight, and can be replaced by working in the setting of any sufficiently saturated elementary extension of the group, for example an ultrapower, which might be familiar to more readers.  For completeness we review ultrapowers of groups and the notation in Section~\ref{reviewofultrapowers}.

\section{Statement of results}

Fix a topological group $\langle G, 0, +, -\rangle$;  that means that $G$ has a Hausdorff topology $\Top$ for which the addition and additive inverse operations are continuous.  While we will write this as an additive group, we will generally not assume that addition is abelian.

Later we will considered \emph{ordered} topological groups, which means that there is a partial order on $G$ which also respects the
operations (but not necessarily the topology).

Call a subset $T$ of $G$ \emph{locally cofinite at} $0$ (LCF0) if
\begin{enumerate}
\item $0\notin T$; and
\item for every open neighborhood $u$ of $0$, $T_u:=T\setminus u$ is finite.
\end{enumerate}

Fix $T_1, T_2,\dots \subseteq G$ locally cofinite at $0$.  Define
\begin{align*}
E_n&:=T_1+\cdots+T_n\\
&=\{a_1+a_2+\cdots+a_n\ :\ a_i\in T_i\}
\end{align*}
(which we can think of as n-term generalized Egyptian fractions generated by $T_1,\dots,T_n$).

\begin{Exa} For every  $i$, let $T_i=\{\frac{1}{m}\ :\ m\in\N\}\subseteq\R$. Then $E_n$ is the set of (classical) n-term Egyptian fractions in $\R$.
\end{Exa}

\begin{Exa} For every  $i$ let  $A_i\subset\R^+$ be finite, $B_i\subset\R^+$ be discrete, and put $T_i=\{\frac{a}{b}\ :\ a\in A_i, b\in B_i\}$.  The set $E_n$ is then the set of \emph{weighted real Egyptian numbers} considered by Nathanson in~\cite{Nath}.  Note that if $\epsilon>0$, $$T_i\setminus (0, \epsilon)=\bigcup_{a\in {A_i}}\{\frac{a}{b} : b\in B_i, b<\frac{a}{\epsilon}\}$$
which is finite, verifying that $T_i$ is LCF0.
\end{Exa}

Sierpinski somehow didn't notice the following:

\begin{Thm}\label{ThmCPCT}For $n>0$ the set $E'_n:=T'_1+\cdots+T'_n$ is compact (where $A'=A\cup\{0\}$)
\end{Thm}

In the case of $G=\R$ this immediately implies Sierpinski's result \cite[Th\'{e}or\`{e}m 2]{Sierpinski}:

\begin{Cor} The set of n-term Egyptian fractions is nowhere dense in $\R$.
\end{Cor}

It likewise implies Nathanson's \cite[Corollary 3]{Nath} more general result:

\begin{Cor} The set of n-term weighted Egyptian real numbers is nowhere dense in $\R$.
\end{Cor}

\begin{proof}[Proof of corollaries] Otherwise $E'_n$ would contain an interval (since it is closed), but it is countable from the definition.
\end{proof}

The proofs of all the results in this section, including Theorem~\ref{ThmCPCT}, are deferred to Section~\ref{proofsection}

We note that while $E'_n$ is compact, $E_n$ might \emph{not} be nowhere dense for arbitrary groups; for example, if $G=\Z$ under addition and $T_i=\{1\}$ then $E_n=\{n\}$ which is dense in $\{n\}$.

However, the following does hold.
\begin{Thm}\label{NWD}Suppose $T_1,\dots,T_n$ are LCF0 and nowhere dense in $G$.  Then $E_n=T_1+\cdots+T_n$ is nowhere dense in $G$.
\end{Thm}

Some other results from \cite{Sierpinski}, \cite{Nath}, and \cite{Ross} still hold in this setting, albeit with slight modification to the statement.

If $g\in E_n$, a \emph{representation} of $g$ in $E_n$ is a tuple $\langle g_1,\cdots, g_n\rangle\in T_1\times\cdots\times T_n$ such that $g=g_1+\cdots+g_n$.  Equivalently, it is a solution in $T_1\times\cdots\times T_n$ to the equation $x_1+\cdots+x_n=g$.

The following now generalizes \cite[Th\'{e}or\`{e}m 1]{Sierpinski}
\begin{Thm}\label{E3finite} Let $g\in E_3$.  Then either $g$ has only finitely many representations in $E_3$, or $g\in T_1\cup T_2\cup T_3$, or $g=0$.
\end{Thm}

For the remainder of this section we assume that $G$ is an \emph{ordered group}, which means that there is a partial order $\le$ on $G$ which honors the operations in the sense that if $a\le b$ and $c\le d$ then $a+c\le b+d$.  As usual, write $a<b$ if $a\le b$ and $a\neq b$.  The \emph{positive cone} of $G$ with respect to this order is $\{g\in G : 0< g\}$.

\begin{Thm}\label{Finite} Let $\langle G, 0, +, -, \le \rangle $ be an ordered topological group, and $T_1,T_2,\dots$ be LCF0 subsets of the positive cone of $G$.  Then the number of representations of any element of $E_n$ is finite.
\end{Thm}

The final results generalize \cite[Th\'{e}or\`{e}m 3]{Sierpinski} (which Sierpinski attributes to Mycielski) and \cite[Theorem 1.2.2]{Ross}.

Say that the order $\le$ on $G$ \emph{respects the topology} provided for any $g\in G$ the set $\SUa{g}=\{x\in G : x<g\}$ is open. The order \emph{weakly respects the topology} provided for any $h<g$ in $G$ there is an open neighborhood $u$ of $h$ such that there is no $x\in u$ with $g<x$.  The first condition implies the second, and they agree on a linearly ordered group.

\begin{Thm}\label{Decrease}Let $\langle G, 0, +, -, \le \rangle $ be an ordered topological group where the order respects the topology, and
let $T_1,T_2,\dots$ be LCF0 subsets of the positive cone of $G$. Any infinite subset $S$ of $E_n$ contains a strictly decreasing sequence.
\end{Thm}

The following is immediate.

\begin{Cor}Let $\langle G, 0, +, -, \le \rangle $ be an ordered topological group where the order respects the topology, and
let $T_1,T_2,\dots$ be LCF0 subsets of the positive cone of $G$.  There does not exist a strictly increasing sequence in $E_n$.
\end{Cor}

\section{Ultrapowers of groups}\label{reviewofultrapowers}

Before proceeding to the proofs of the results in the previous section, we review the definition and some properties of the ultrapower of a group.  As everything in this section is well-known or an easy exercise, no proofs are included (but see the comments at the end of the section).

\subsection{Definitions}
Let $I$ be an infinite set and $\SU$ be an ultrafilter on $I$.  That is, $\SU\subseteq\PS(I)$ satisfies: (i)~$\emptyset\not\in\SU$;
(ii)~If $A\subseteq B$ and $A\in\SU$ then $A\in\SU$; (iii)~If $A, B\in\SU$ then $A\cap B\in\SU$; and
(iv) For every $A\subseteq I$ either $A\in\SU$ or $\compl{A}\in \SU$.

For example, if $a\in I$ is a fixed element then $\{A\subseteq I : a\in A\}$ is an ultrafilter on $I$, called a \emph{principal} ultrafilter.  Nonprincipal ultrafilters are usually ``constructed" using Zorn's Lemma.

If $G$ is a nonempty \emph{set} then denote by $G^I$ the set of functions from $I$ into $G$.  If $f,h\in G^I$ write $f\equiv_\SU h$ if
$\{i\in I : f(i)=h(i)\}\in\SU$.  Say that $f=g$ \emph{almost surely}, or as.

By the properties of $\SU$, $\equiv_\SU$ is an equivalence relation on $G^I$.
Denote the equivalence class of $f$ by $f/\SU$.  Write $\star{G}=G^I/\SU$ for the set of all such equivalence classes.
If $A\subseteq G$ write $\star{A}=A^I/\SU$.  Note this can be identified with $\{f/\SU : f\in G^I, f(i)\in A\textrm{ as}\}$.
More generally, if $A\subseteq G^n$ write $\star{A}=\{\langle f_1/\SU,\dots,f_n/\SU\rangle : f_1,\dots,f_n\in G^I, \langle f_1(i),\dots,f_n(i)\rangle\in A\textrm{ as} \}$.  Note $\star{A}\subseteq (\star{G})^n$.

If $g\in G$ (or, more generally, $G^n$) write $\star{g}$ for the unique element of $\star{\{g\}}$
If $g\neq h$ then $\star{g}\neq\star{h}$; thus $*$ is an embedding of $G$ into $\star{G}$, so we usually write $g$ instead of $\star{g}$ for $g\in G$.

Now, if $G$ is a \emph{group} with binary operation $+$, define $\star{+}:\star{G}\times \star{G}\to\star{G}$ by
$f/\SU\star{+}g/\SU=(f+g)/\SU$.
This is well-defined, and extends $+$ in the sense that if $a,b\in G$ then $\star{(a+b)}={a}\star{+}{b}$.
Likewise any other function from $G^n$ to $G$ (such as $-$), and any relation (such as a partial order $\le$)
on $G^n$ can be extended in a similar way.  It is reasonable to write $+, -, \le$ instead of $\star{+}, \star{-}$, and $\star{\le}$.

If a \emph{set} $G$ is equipped with a topology $\Top$, for $g\in G$ let:
$$\mu(g)=\bigcap\limits_{g\in u\in\Top}\star{u}=\{h\in\star{G} : h\in\star{u}\textrm{ for every open neighborhood }u\textrm{ of }g\}$$

\begin{Prop}If $\Top$ is Hausdorff then for all $a\neq b\in G$, $\mu(a)\cap\mu(b)=\emptyset$
\end{Prop}
\noindent (In fact, the converse is true for most interesting ultrafilters, but false for, eg, principal ultrafilters.)

Write $$NS(G)=\bigcup\limits_{g\in G}\mu(g)=\{h\in\star{G} : h\in\mu(g)\textrm{ for some }g\textrm{ in }G\}.$$Define $\st : NS(G)\to G$ by $\st(h)=$ the unique $g\in G$ with $h\in\mu(g)$.  For $g,h\in NS(G)$, if $\st(g)=\st(h)$ write $g\approx h$.

\subsection{Properties}
Most of the significant properties of $\star{G}$ are listed in the following theorem.

\begin{Thm}\label{thm:properties} For any \emph{set} $G$ there exists an ultrafilter $\SU$ on a set $I$ with the following properties.
\begin{enumerate}[label=(S\arabic*)]
\item $* : \PS(G^n)\to\PS((\star{G})^n)$ respects finite Boolean operations: $\star{\emptyset}=\emptyset; \star{(G^n)}=(\star{G})^n$;
if $a\in A$ then $a\in\star{A}$; if $A\subseteq B$ then $\star{A}\subseteq \star{B}$;
$\star(A\cup B)=\star{A}\cup \star{B}; \star(A\cap B)=\star{A}\cap \star{B}; \star(A\setminus B)=\star{A}\setminus \star{B}$.

\item If $A$ is finite then $A=\star{A}$. If $A$ is infinite then $A\subsetneq\star{A}$.

\item If $A_j\subseteq G^n$ for $j$ in some index set $J$ and  $\{A_j\}_{j\in J}$ has the finite intersection property
then $\bigcap_{j\in J}\star{A_j}\neq\emptyset$

\end{enumerate}
\noindent If in addition $G$ is a group with operation $+$ and identity $0$ then also:
\begin{enumerate}[label=(G\arabic*)]
\item $\star{G}$ is a group with operation $\star{+}$ and identity $\star{0}$.
\item $*$ is an isomorphism of $G$ into $\star{G}$.
\item If $A, B\subseteq G$ then $\star{(A+B)}=\star{A}+\star{B}$ and $\star{(-A)}=-\star{A}$
\end{enumerate}

\noindent If in addition $G$ is an ordered group with partial order $\le$, then:

\begin{enumerate}[label=(A\arabic*)]
\item Suppose $A\subseteq G^n$ is the set of elements satisfying a given group equation or inequality (possibly with parameters from $G$).  Then $\star{A}$ is the subset of $\star{G^n}$ satisfying the same equation or inequality.
\end{enumerate}

\noindent If $G$ is a topological group with (Hausdorff) topology $\Top$, then:
\begin{enumerate}[label=(T\arabic*)]
\item $NS(G)$ is a subgroup of $\star{G}$.
\item $\st$ is a group homomorphism.
\item For $u\subseteq G$ the following are equivalent: (i)~$u$ is open; (ii)~$\st^{-1}(u)\subseteq\star{u}$; (iii)~for every $g\in u$, $\mu(g)\subseteq\star{u}$.
\item For $K\subseteq G$ the following are equivalent: (i)~$K$ is compact;  (ii)~$\star{K}\subseteq\st^{-1}(K)$; (iii)~for every $h\in\star{K}$ there is a $g\in K$ with $h\in\mu(g)$.
\end{enumerate}

\end{Thm}

\subsection{Comments}
Most of the results of Theorem~\ref{thm:properties}, sometimes with a slight rewording, should be familiar to anyone familiar with ultrapowers of structures.  They are covered in any modern introduction to nonstandard analysis (for example \cite{Edinburgh} or \cite{Keisler}) or model theory (\cite{ChangKeisler}), or in stand-alone introductions to applications of ultrafilters in mathematics (such as \cite{BellSlomson} or \cite{Tao})

Many of the results are consequences of the L{\o}s Theorem (or \emph{transfer principal} in nonstandard analysis), and are true for any ultrafilter (even principal ones).  Property S3 is a weak form of $\kappa^+-$\emph{saturation}, where $\kappa$ is at least as large as any index set $J$ that might arise, but in this paper $\kappa$ is never larger than the cardinality of the power set of $G$. See \cite{ChangKeisler} for a discussion of conditions on ultrafilters that give rise to $\kappa^+-$saturated ultrapowers.

The definitions of $\mu(g)$ (the \emph{monad} of $g$), $NS(G)$ (the elements \emph{nearstandard} to $G$), and $\st$ (the \emph{standard part} map) can be made in any Hausdorff topological space, not just topological group, and properties (T3) and (T4) still hold.  Property (T4), sometimes called the \emph{Robinson compactness criterion}, is quite powerful and is the key to the simplicity of the proof of Theorem~\ref{ThmCPCT} in this paper.  Property (T2) codifies the continuity of the group operations.

For an example of A1, fix $a\in G$.  Then: $$\star{\{g\in G : g\le a\}}=\{g\in \star{G} : g\le a\}$$

\section{Proofs}\label{proofsection}

Start with a simple property of $\st$.

\begin{Prop}\label{STD0}For LCF0 $T$, if $x\in\star{T}$ then either $x\in T$ or $x\approx 0$.
\end{Prop}

\begin{proof}Suppose $x\notin T$.  If $u$ is any open neighborhood of $0$ then $x\notin T\setminus u = \star{T}\setminus\star{u}$ (by Property~S1), so $x\in\star{u}$.
\end{proof}

\begin{proof}[Proof of Theorem~\ref{ThmCPCT}]
Suppose
\begin{align*}x&=x_1+\cdots+x_n\\
&\in\star{(T'_1+\cdots+T'_n)}\\
&=\star{T'_1}+\cdots+\star{T'_n} \qquad \textrm{(by Property G3)}
\end{align*}
Since each $\star{T_i}\subseteq NS(G)$ by the previous proposition, $\star{T'_1}+\cdots+\star{T'_n}\subseteq NS(G)$ by Property T1.  Moreover, suppose
\[y_i=\left\{
                  \begin{array}{ll}
                    x_i, & \hbox{$x_i\in T_i$;} \\
                    0, & \hbox{otherwise.}
                  \end{array}
                \right.\]
Then $\st(x)=y_1+\cdots+y_n\in E'_n$.  Compactness of $E'_n$ now follows by Property T4.
\end{proof}

\begin{proof}[Proof of Theorem~\ref{E3finite}] Suppose $g\in E_3, g\neq 0$ has infinitely many representations. By Property~S2 there is a representation $g=g_1+g_2+g_3$ in $\star{T_1}+\star{T_2}+\star{T_3}$ with at least one $g_i\not\in T_i$.  There are three cases.
Case 1. $g_i\not\in T_i$ for all $i$.  Then $g=\st(g)=\st(g_1)+\st(g_2)+\st(g_3)=0$, a contradiction.
Case 2. $g_i\not\in T_i$ for exactly two of the three values of $i$.  Denoting the remaining value by $j$, by additivity of $\st$ and Proposition~\ref{STD0} we get $g=\st(g)=\st(g_1)+\st(g_2)+\st(g_3)=g_j\in T_j$.
Case 3. $g_i\not\in T_i$ for exactly one of the three values of $i$.  Denoting this value by $j$ and solving for $g_j$, we get
$$g_j=\left\{
  \begin{array}{ll}
    g-(g_2+g_3), & \hbox{j=1;} \\
    -g_1+g-g_3, & \hbox{j=2;} \\
    -(g_1+g_2)+g, & \hbox{j=3.}
  \end{array}\right.$$In any of these alternatives we have an element of $\star{T_j}\setminus T_j$, which must be a nonzero element of $\mu(0)$, equal to an element of $G$.  This can't happen.
\end{proof}

\begin{Prop}\label{STDless}Suppose $G$ is an ordered group and $T_1,T_2,\dots$ are LCF0 subsets of the positive cone of $G$.  (i)~If $x=x_1+\cdots+x_n\in \star{E_n}$ then $\st(x)\le x$.  (ii)~If in addition some $x_i\not\in T_i$ then  $\st(x)<x$
\end{Prop}

\begin{proof}(i)~follows from Proposition~\ref{STD0} and additivity of $\st$.
\begin{align*}
\textrm{(ii)}~\st(x)&=\st(x_1+\cdots+x_n)\\
&=\st(x_1+\cdots+x_{i-1})+\st(x_i)+\st(x_{i+1}+\cdots+x_n)\\
&\le (x_1+\cdots+x_{i-1})+\st(x_i)+(x_{i+1}+\cdots+x_n) \quad\textrm{(by (i))} \\
&< (x_1+\cdots+x_{i-1})+x_i+(x_{i+1}+\cdots+x_n) \quad\textrm{(by Proposition~\ref{STD0})}\\
&=x \qedhere
\end{align*}
\end{proof}

\begin{proof}[Proof of Theorem~\ref{Finite}]
Otherwise there is a $g\in E_n$ with infinitely many representations.  We may suppose $n$ is as small as possible. By Property~S2 there is a representation $g=g_1+\cdots+g_n$ in $\star{T_1}+\cdots+\star{T_n}$ with at least one $g_i\not\in T_i$.  By Proposition~\ref{STDless} $g=\st(g)<g$, a contradiction.
\end{proof}

\begin{Lem}\label{NWD0}Let $T, E$ be NWD subsets of $G$ with $T$ be locally cofinite at $0$. Then $E+T$ is NWD in $G$.
\end{Lem}

\begin{proof}
Suppose $E+T$ is not NWD.  There is then a nonempty open subset $v$ of the interior of the closure $\overline{E+T}$.  Since $E$ is NWD, we may assume $v\cap\overline{E}=\emptyset$.  Let $c\in v$.  For $u, w\in\Top$ let
$$A_{u,w}=u\cap(E+T)\cap\compl{\overline{E+T_w}}$$
If $c\in u\subseteq v$ and $0\in w$ then $E+T_w$ is a finite union of NWD sets, hence itself NWD, and $u\setminus\overline{E+T_w}$ is a nonempty open subset of $\overline{E+T}$, so must contain an element of ${E+T}$.  It follows that $A_{u,w}$ is nonempty.  Since $A_{u,w}$
is decreasing in both $u$ and $w$, $\{A_{u,w} : u,w\in\Top, 0\in w, c\in u\subseteq v\}$ has the finite intersection property, so there is a $g$ in $\bigcap_{u,w\in\Top, 0\in w, c\in u\subseteq v}\star{A_{u,w}}$.

Then $g\in\star{(E+T)}$, so $g=e+t$ for some $e\in\star{E}, t\in\star{T}$.  Since for any neighborhood $w$ of $0$ $t\notin\star{T_w}$, $t\approx 0$.  Evidently $g\in\mu(c)$, so $c\approx e+t\approx e$.  Thus $\star{E}\cap\mu(c)\neq\emptyset$, so $c\in\overline{E}$, a contradiction.
\end{proof}

\begin{proof}[Proof of Theorem~\ref{NWD}] This follows from Lemma~\ref{NWD0} and induction.\end{proof}

\begin{proof}[Proof of Theorem~\ref{Decrease}]Let $S$ be an infinite subset of $E_n$.  Since $S$ is infinite, there is an element $s\in \star{S}\setminus S$.  Let $g=\st(s)$.  By Proposition~\ref{STDless}, $g<s$.  Define a decreasing sequence $g_m$ in $S$ as follows.

The set $A_1=\{x\in S : g<x\}$ is nonempty, since $s\in\star{A_1}$.  Let $g_1\in A_1$.

Given $g_m\in S$ with $g<g_m$, let $A_{m+1}=\{x\in S : g<x\}\cap\SUa{g_m}$.  Since $g\in \SUa{g_m}$ and $\SUa{g_m}$ is open, $s\in\star{A_{m+1}}$, so $A_{m+1}\neq\emptyset$.  Let $g_{m+1}$ be an element of $A_{m+1}$.

Then the sequence $g_m$ is a decreasing sequence in $S$ by construction.
\end{proof}

\end{document}